\documentclass{drna_preprint}

\newenvironment{quiet-proof}{}{\qed}
\newenvironment{Remark}{\noindent\textsl{Remark.}}{}
\newenvironment{Remarks}{\noindent\textsl{Remarks.}}{}

\newcommand{\Natural}{\mathbb N}
\newcommand{\Real}{\mathbb R}
\newcommand{\Sphere}{\mathbb S}

\newcommand{\Abs}[1]{\left\vert #1 \right\vert}
\newcommand{\Card}[1]{\abs{#1}}

\newcommand{\To}{\rightarrow}
\newcommand{\Order}{\operatorname{O}}

\newcommand{\potential}{u}
\newcommand{\Potential}{U}

\newcommand{\Energy}{\operatorname{E}}

\newcommand{\Discrepancy}{\operatorname{\mathcal{D}}}
\newcommand{\Disc}{\operatorname{\delta}}
\newcommand{\MinDistBound}{\Delta}

\newcommand{\dist}{\operatorname{\mathtt{dist}}}
\newcommand{\diam}{\operatorname{\mathtt{diam}}}
\newcommand{\Ball}{B}
\newcommand{\Integral}{\mathcal{I}}
\newcommand{\MeanPotential}{\Phi}
\newcommand{\Measure}{\lambda}
\newcommand{\NMeasure}{\sigma}

\newcommand{\RadiusH}{\rho}
\newcommand{\Volume}{\mathcal{V}}
\newcommand{\noemph}[1]{#1}

\begin{document}

\title{Discrepancy, separation and Riesz energy of finite point sets on compact connected Riemannian manifolds}

\author{Leopardi}{Paul}{a}

\affiliation{a}{Mathematical Sciences Institute, Australian National University, Canberra, Australia}

\firstpage{1}

\maketitle

%
\begin{abstract}
\noindent
On a smooth compact connected $d$-dimensional Riemannian manifold $M$,
if $0 < s < d$ then
an asymptotically equidistributed sequence of finite subsets of $M$
that is also well-separated yields a sequence of Riesz $s$-energies that converges
to the energy double integral, 
with a rate of convergence depending on the geodesic ball discrepancy.
This generalizes a known result for the sphere.
%
\end{abstract}

%
\section{Introduction and Main Results}

This paper arises from a remark at the end of the related paper~\cite{Leo13} on separation,
discrepancy and energy on the unit sphere, that the results of Bl\"umlinger~\cite{Blu90}
could be used to generalize the results given there.
The main result of that paper is that,
for the unit sphere $\Sphere^d \in \Real^{d+1},$ with $d \geqslant 2,$ 
if $0 < s < d$ then
an asymptotically equidistributed sequence of spherical codes 
that is also well-separated yields a sequence of Riesz $s$-energies that converges
to the energy double integral, 
with the rate of convergence depending on the spherical cap discrepancy~\cite[Theorem 1.1]{Leo13}.
This paper generalizes that result to the setting of the volume measure on a Riemannian manifold,
with a potential based on geodesic distance.

The relationships between discrepancy and energy of measures on a manifold have been studied for a long time,
in various settings, and there is an extensive literature, including works by
Benko, Damelin, Dragnev, Hardin, Hickernell, Ragozin,
Saff, Totik,  Zeng and many others~\cite{BenDD12,DamHRZ10,HarS05,SafT97}.
(See also the bibliography of the related work on the unit sphere~\cite{Leo13}
for further references specific to that setting.) 
Many of these works have concentrated on equilibrium measure~\cite{BenDD12,DamHRZ10,SafT97}
and on manifolds embedded in Euclidean space,
with a potential based on Euclidean distance~\cite{BenDD12,HarS05}.
In contrast, this paper focuses on the volume measure on a Riemannian manifold,
with a potential based on geodesic distance.
As a consequence, many results from the literature, concerning, e.g. the support of
an equilibrium measure~\cite{BenDD12} do not apply here.
Instead, this paper takes the approach of translating the methods used in~\cite{Leo13}
to the setting of Riemannian geometry.

For $d \geqslant 1$ let $M$ be a smooth connected $d$-dimensional Riemannian manifold,
without boundary, with metric $g$ and geodesic distance $\dist$, 
such that $M$ is compact in the metric topology of $\dist$.
Let $\diam(M)$ be the \emph{diameter} of $M,$
the maximum geodesic distance between points of $M.$
Let $\Measure_M$ be the volume measure on $M$ given by 
the volume element corresponding to the metric $g$. 
Since $M$ is compact, it has finite diameter and finite volume.
Let $\NMeasure_M$ be the probability measure $\Measure_M / \Measure_M(M)$ on $M$.
For the remainder of this paper, all compact connected Riemannian manifolds are 
assumed to be finite dimensional, smooth and without boundary, unless otherwise noted.

This paper concerns infinite sequences $\mathcal{X} := (X_1, X_2, \ldots)$
of finite subsets of the manifold $M.$
Each such finite subset is called an $M$-\emph{code}, 
by analogy with spherical codes, which are finite subsets of the unit sphere $\Sphere^d.$
A sequence $(X_1, X_2, \ldots)$ whose corresponding sequence
of cardinalities $(\Card{X_1}, \Card{X_2}, \ldots)$ diverges to $+\infty$
is called a \emph{pre-admissible} sequence of $M$-codes.

For any probability measure $\mu$ on $M$, the \emph{normalized ball discrepancy} is
\begin{align*}
\Discrepancy (\mu)
&:=
\sup_{x \in M,\ r > 0} \Abs{\mu\big(\Ball(x,r)\big) - \NMeasure_M\big(\Ball(x,r)\big)},
\end{align*}
where 
$\Ball(x,r)$ is the geodesic ball of radius $r$ about the point $x$~\cite{Blu90,DamG03}.

An $M$-code $X$ with cardinality $\Card{X}$ has a corresponding probability measure $\NMeasure_X$ and
normalized ball discrepancy $\Discrepancy(X),$
where for any measurable subset $S \subset M,$ 
\begin{align*}
\NMeasure_X(S) &:= \Card{S \cap X}/\Card{X},
\end{align*}
and 
\begin{align*}
\Discrepancy(X) &:= \Discrepancy(\NMeasure_X) 
= \sup_{y \in M,\ r > 0} \Abs{\Card{\Ball(y,r) \cap X}/\Card{X} - \NMeasure_M\big(\Ball(y,r)\big)}.
\end{align*}
It is easy to see that $\Discrepancy(X) \geqslant 1/\Card{X},$ since for any $x \in X,$ 
$\NMeasure_M\big(\Ball(x,r)\big)$ can be made arbitrarily small by taking $r \To 0,$
while $\NMeasure_X\big(\Ball(x,r)\big)$ must always remain at least $1/\Card{X},$
since the ball $\Ball(x,r)$ contains the point $x \in X.$

A pre-admissible sequence $\mathcal{X} := (X_1, X_2, \ldots),$ of $M$-codes 
with corresponding cardinalities $N_{\ell}:=\Card{X_{\ell}}$
is \emph{asymptotically equidistributed}~\cite[Remark 4, p. 236]{DamG03},
if the normalized ball discrepancy is bounded above as per
\begin{align}
\Discrepancy (X_{\ell}) < \Disc(N_{\ell}),
\label{discrepancy-bound-def}
\end{align}
where $\Disc : \Natural \To (0,1],$ is a positive decreasing function with $\Disc(N) \To 0$ as $N \To \infty$.

By the \emph{minimum geodesic distance} of a code $X$, we mean the minimum, 
over all pairs $(x,y)$ of distinct code points in $X,$ of the geodesic distance $\dist(x,y).$
The pre-admissible sequences of $M$-codes 
of most interest for this paper are those such that the minimum geodesic distance
is bounded below as per
\begin{align}
\dist(x,y) &> \MinDistBound(N_{\ell})
\quad\text{for all}\ x, y \in X_{\ell},
\label{separation-bound-def}
\end{align}
where $\MinDistBound : \Natural \To (0,\infty),$ is a positive decreasing function with $\MinDistBound(N) \To 0$ as $N \To \infty$.

Flatto and Newman~\cite[Theorem 2.2]{FlaN77}, 
in the case where the manifold $M$ is $C^4$ rather than smooth,
showed that there exists a positive constant $\gamma,$ depending on $M,$
such that a sequence of $M$-codes exists with $\MinDistBound(N_{\ell}) = \gamma N_{\ell}^{-1/d}$.
In the case of smooth manifolds, as treated here, we call such a sequence of $M$-codes
\emph{well separated} with \emph{separation constant} $\gamma.$

An easy area argument shows that the order $\Order(N^{-1/d})$ is best possible,
in the sense that, for any sequence of $M$-codes,
any applicable lower bound of the form~\eqref{separation-bound-def}
is itself bounded above by
\begin{align*}
\MinDistBound(N_{\ell}) &= \Order(N_{\ell}^{-1/d}),
%
\end{align*}
(as $\ell \To \infty$).

For the purposes of this paper, we define an \emph{admissible sequence} of $M$-codes
to be a pre-admissible sequence $\mathcal{X},$ such that a discrepancy function $\Disc$ 
and a separation function $\MinDistBound$ exist,
satisfying the bounds~\eqref{discrepancy-bound-def} and~\eqref{separation-bound-def} respectively.

For $0 < s < d$, the \emph{normalized Riesz $s$-energy} of an $M$-code $X$ is
$\Energy_X\,\Potential^{(s)}$,
where $\Energy_X$ is the normalized discrete energy functional
\begin{align*}
\Energy_X\,\potential &:=
\frac{1}{\Card{X}^2} \sum_{x \in X} \mathop{\sum_{y \in X}}_{y\neq x} \potential \left(\dist(x,y)\right),
\end{align*}
for $u : (0,\infty) \To \Real,$
and $\Potential^{(s)}(r) := r^{-s},$ the Riesz potential function,
for $r \in (0,\infty).$

The corresponding normalized continuous energy functional is given by the double integral~\cite{DamLRS09,HarR03}
\begin{align*}
\Energy_M \potential 
&:=
\int_M \int_M \potential \left(\dist(x,y)\right) d\NMeasure_M(y)\,d\NMeasure_M(x).
\end{align*}

The main result of this paper is the following theorem.
\begin{theorem}\label{manifold-discrepancy-separation-energy-theorem}
Let $M$ be a compact connected $d$-dimensional Riemannian manifold.
If $0 < s < d$ then,
for a well separated admissible sequence $\mathcal{X}$ of $M$-codes,
the normalized Riesz $s$-energy converges to the energy double integral
of the normalized volume measure $\NMeasure_M$ as $\Card{X_{\ell}} \To \infty$. 
The rate of convergence of the energy difference is of order $\Disc(\Card{X_{\ell}})^{(1-s/d)/(d+2-s/d)},$
where $\Disc(\Card{X_{\ell}})$ is the upper bound on the geodesic ball discrepancy of $X_{\ell}$
used to satisfy the admissibility condition.
That is,
\begin{align}
\Abs{ \big(\Energy_{X_{\ell}} - \Energy_M\big)\,\Potential^{(s)} } 
&= \Order\big(\Disc(\Card{X_{\ell}})^{(1-s/d)/(d+2-s/d)}\big),
\label{eq-manifold-discrepancy-separation-energy-estimate}
\intertext{and therefore}
\Abs{ \big(\Energy_{X_{\ell}} - \Energy_M\big)\,\Potential^{(s)} } &\To 0
\quad \text{as}\ \Card{X_{\ell}} \To \infty.
\notag
\end{align}
\end{theorem}

The proof of Theorem~\ref{manifold-discrepancy-separation-energy-theorem} is given in 
Section~\ref{proof-section} below.
This proof is similar to that of Theorem 1.1 in the corresponding paper on the unit sphere~\cite{Leo13},
except for two key points of difference:

\begin{enumerate}
\item 
The \noemph{normalized mean potential function}
\begin{align*}
\MeanPotential_M^{(s)}(x) &:= \int_M \Potential^{(s)} \left(\dist(x,y)\right) d\NMeasure_M(y)
\end{align*}
may vary with $x$, unlike the case of the sphere, where the corresponding mean potential function is a constant.

\item 
The volume of a geodesic ball in general does not behave in exactly the same way as the volume of a spherical cap.
Luckily the appropriate estimate is good enough to obtain the result.
\end{enumerate}

Bl\"umlinger~\cite[Lemma 2]{Blu90} gives an estimate related to the 
\noemph{Bishop-Gromov inequality}~\cite[11.10, pp. 253--257]{BisC64} 
\cite[Lemma 5.3bis pp. 65--66]{Gro81}~\cite[Lemma 5.3bis pp. 275--277]{Gro99}.
In the notation used here, Bl\"umlinger's estimate states:

%
Let $M$ be a compact connected $d$-dimensional Riemannian manifold.
Then
\begin{align*}
\Abs{ \frac{\Measure_M\big(\Ball(x,r)\big)}{\Volume_d(r)}-1 } &= \Order(r^2)
\end{align*}
(as $r \To 0$) uniformly in $M$,
where $x \in M$ and $\Volume_d(r)$ is the volume of the Euclidean ball of radius $r$ in $\Real^d$.
That is, the unnormalized volume of a small enough geodesic ball in $M$ 
is similar to the volume of a ball of the same radius in $\Real^d$,
to the order of the square of the radius.

\begin{Remarks}
\begin{enumerate}
\item 
Bl\"umlinger's paper treats smooth compact connected 
Riemannian manifolds $M$ whose Riemannian measure $\Measure$
is such that $\Measure(M) = 1$~\cite[p. 178]{Blu90}, but it is clear
from the statement of Lemma 2 and its proof that the result also applies to $M$
where $\Measure(M)$ is any positive value.  
\item 
Flatto and Newman~\cite[Theorem 2.3 and Remarks]{FlaN77} prove a similar result,
with an estimate of order $\Order(r)$ for $C^4$ manifolds, 
and order $\Order(r^2)$ for $C^5$ manifolds.
\end{enumerate}
\end{Remarks}

The proof of Lemma 2 in Bl\"umlinger's paper~\cite{Blu90} makes it clear that 
the order estimate is valid for $r < R_{0},$
where $R_{0}$ is the \emph{injectivity radius} of 
$M$~\cite[Lemma 3, Section 8.2, p. 153]{BisC64}~\cite[Definition 4.12, p. 110]{Sak96}.
Thus, Bl\"umlinger's estimate can be restated as the following result.

\begin{lemma}\label{Blumlinger-estimate-lemma}
Let $M$ be a compact connected $d$-dimensional Riemannian manifold,
and let $R_{0}$ be the injectivity radius of $M.$
There exists a real positive constant $C_{0}$ such that for $r \in (0,R_{0})$
and any $x \in M,$
\begin{align}
\Abs{\frac{\Measure_M\big(\Ball(x,r)\big)}{\Volume_d(r)}-1} &\leqslant C_{0} r^2.
\label{eq-Blumlinger-estimate}
\end{align}
\end{lemma}

\section{Notation and results used in the proof of Theorem~\ref{manifold-discrepancy-separation-energy-theorem}}
\label{preliminary-lemmas-section}

The proof of Theorem~\ref{manifold-discrepancy-separation-energy-theorem} 
needs some notation and a few more results, which are stated here.

Firstly, note that this paper, in common with the previous paper~\cite{Leo13}
uses ``big-Oh'' notation with \emph{inequalities} 
in a somewhat unusual way, to avoid a proliferation of unknown constants.
For a positive real function $h,$
when we say that
\begin{align*}
f(n) &\leqslant g(n) + \Order\big(h(n)\big) \quad\text{as}\ n \To \infty,
\intertext{we mean that there exist positive constants $C$ and $M$ such that}
f(n) &\leqslant g(n) + C\big(h(n)\big) \quad\text{for all}\ n \geqslant M.
\end{align*}
If more than one $\Order$ expression is used in an inequality, the implied constants
may be different from each other.

Also, when we say that
\begin{align*}
\Abs{f(n)} &= \Theta(h(n)) 
 \quad\text{as}\ n \To \infty,
\intertext{we mean that there exist positive constants $c < C$ and $M$ such that}
c\big(h(n)\big) &\leqslant \Abs{f(n)} \leqslant C\big(h(n)\big) \quad\text{for all}\ n \geqslant M.
\end{align*}

The next three results follow from Bl\"umlinger's estimate.

\begin{lemma}\label{Small-ball-estimate-lemma}
Let $M$ be a compact connected $d$-dimensional Riemannian manifold.
There is a radius $R_{1} > 0$ and parameters $0< C_{L} < C_{H},$ 
depending on $R_{1},$
such that for all $x \in M$ and all $r \in (0,R_{1})$,
\begin{align}
C_{L}\,r^d &\leqslant \NMeasure_M\left(\Ball(x,r)\right) \leqslant C_{H}\,r^d.
\label{eq-small-ball-estimate}
\end{align}
The ratio $C_{H}/C_{L}$ can be made arbitrarily close to 1 by taking $R_{1}$ small enough.
\end{lemma}
\begin{proof} 

Let $R_{0} > 0$ be the injectivity radius of $M,$ 
so that Bl\"umlinger's estimate~\eqref{eq-Blumlinger-estimate} holds for $r \in (0,R_{0}).$
Note that for each $d$, $\Volume_d(r) = c_d r^d,$ where $c_d := \Volume_d(1) > 0.$
It follows that for all $r \in (0, R_{0})$ the estimate
\begin{align}
c_d r^d (1 - C_{0} r^2 ) &\leqslant \Measure_M\big(\Ball(x,r)\big) \leqslant c_d r^d (1 + C_{0} r^2 )
\label{eq-ball-unnormalized-estimate}
\end{align}
holds for some $C_{0} > 0.$
Let $R_{1} \in (0,R_{0})$ satisfy $C_{0} R_{1}^2 < 1$ so that 
the lower bound in the estimate~\eqref{eq-ball-unnormalized-estimate} is positive for $r \in (0,R_{1}]$.
It follows that for all $r \in (0, R_{1}),$
\begin{align*}
0 
&< 
\frac{c_d (1 - C_{0} R_{1}^2)}{\Measure_M(M)} r^d  
\leqslant 
\NMeasure_M\big(\Ball(x,r)\big) 
\leqslant 
\frac{c_d (1 + C_{0} R_{1}^2)}{\Measure_M(M)} r^d.
\end{align*}
The estimate~\eqref{eq-small-ball-estimate} therefore holds for $R_{1}$ as above,
$C_{L} :=  c_d (1 - C_{0} R_{1}^2 )/ \Measure_M(M),$
and $C_{H}  :=  c_d (1 + C_{0} R_{1}^2 )/ \Measure_M(M).$
In this case,
\begin{align*}
\frac{C_{H}}{C_{L}} 
&=
\frac{1 + C_{0} R_{1}^2}{1 - C_{0} R_{1}^2}
\quad \To 1,
\quad
\text{as} 
\quad
R_{1} \To 0. &
\end{align*}
\end{proof}

\begin{lemma}\label{Large-ball-estimate-lemma}
Let $M$ be a compact connected $d$-dimensional Riemannian manifold.
There are positive constants $C_{bot} < C _{top}$ depending only on $M,$
such that for all $x \in M$ and all $r \in (0,\diam(M)]$,
\begin{align}
C_{bot}\,r^d &\leqslant \NMeasure_M\left(\Ball(x,r)\right) \leqslant C_{top}\,r^d.
\label{eq-large-ball-estimate}
\end{align}
\end{lemma}
\begin{proof} ~

Let $R_{1},$ $C_L$ and $C_H$ be as in the statement of Lemma~\ref{Small-ball-estimate-lemma}
and its proof.

For $r \in [R_1, \diam(M)],$ the following inequalities hold:
\begin{align*}
C_L \frac{R_{1}^d}{\diam(M)^d} r^d
&\leqslant
C_L R_{1}^d
\leqslant
\NMeasure_M\left(\Ball(x,R_{1})\right)
\leqslant
\NMeasure_M\left(\Ball(x,r)\right)
\leqslant
1
\leqslant
\frac{1}{R_{1}^d} r^d.
\end{align*}
Thus the inequality~\eqref{eq-large-ball-estimate} is satisfied with
$C_{bot} := C_L R_{1}^d / \diam(M)^d$ and 
$C_{top} := \max( C_H, R_{1}^{-d} ).$ 

\end{proof}

\begin{lemma}\label{Ball-packing-argument-lemma}
Let $M$ be a compact connected $d$-dimensional Riemannian manifold.
For $x \in M$ and real $r > t > 0$ let $n_M(x,r,t)$ be the maximum number of 
disjoint geodesic balls of radius $t$ that can be contained
in the ball $\Ball(x, r).$
Then there is a constant $C_{2}$
such that for all $x \in M,$ $r \in (0,\diam(M)),$ and $q \in (0,r),$
\begin{align}
n_M(x,r+q/2,q/2) &\leqslant C_{2}\,(r/q)^d.
\label{eq-ball-packing-argument-estimate}
\end{align}
\end{lemma}
In other words, for real positive $r$, for $0 < q < r,$
the maximum number of geodesic balls of radius $q/2$ that can be contained in
a geodesic ball of radius $r+q/2$ is of order $\Order(r/q)^d,$ uniformly in $M$. 

\begin{proof} ~

The total volume of the small balls cannot be greater than the volume of the large ball containing them.
Using Lemma~\ref{Large-ball-estimate-lemma}, it therefore holds for $0 < q < r \leqslant \diam(M) - q/2$ that
\begin{align*}
n_M(x,r+q/2,q/2)
&\leqslant
\frac{\max_{y \in M} \NMeasure_M\big(\Ball(y,r+q/2)\big)}{\min_{z \in M} \NMeasure_M\big(\Ball(z,q/2)\big)}
\\
&\leqslant
 2^d \frac{C_{top}}{C_{bot}} \left(1+\frac{q}{2r}\right)^d (r/q)^d
\leqslant
 3^d \frac{C_{top}}{C_{bot}} (r/q)^d.
\\
\end{align*}
For $r > \diam(M) - q/2,$
the following relationships therefore hold:
\begin{align*}
n_M(x,r+q/2,q/2) 
&= n_M(x,\diam(M),q/2) 
\\
&\leqslant 3^d \frac{C_{top}}{C_{bot}} ((\diam(M) - q/2)/q)^d 
\leqslant 3^d \frac{C_{top}}{C_{bot}} (r/q)^d.
\end{align*}
Thus~\eqref{eq-ball-packing-argument-estimate} holds with $C_{2} := 3^d C_{top}/C_{bot}.$

\end{proof}

The remaining lemmas in this Section as well as the proof of 
Theorem~\ref{manifold-discrepancy-separation-energy-theorem}
make use of the following definitions.

For $x \in M,$ real radius $r > 0,$ and integrable $f : \Ball(x,r) \To \Real,$ 
the \noemph{normalized integral of $f$ on the geodesic ball $\Ball(x,r)$} is
\begin{align*}
\Integral_{\Ball(x,r)} f &:= \int_{\Ball(x,r)} f(y)\,d\NMeasure_M(y).
\end{align*}
For integrable $f : M \To \Real$ the \noemph{mean of $f$ on $M$} is
\begin{align*}
\Integral_M f &:= \int_M f(y)\,d\NMeasure_M(y).
\end{align*}
For a function $f : M \To \Real$ that is finite on the $M$-code $X$,
the \noemph{mean of $f$ on $X$} is
\begin{align*}
\Integral_X f &:= \int_M f(y)\,d\NMeasure_X(y) = \frac{1}{\Card{X}} \sum_{y \in X} f(y).
\end{align*}

For an $M$-code $X$, a point $x \in M$ and a measurable subset $S \subset M,$ 
the \noemph{punctured normalized counting measure of $S$ with respect to $X$, excluding $x$} is
\begin{align*}
\NMeasure_X^{[x]}(S) &:= \Card{S \cap X \setminus \{x\}}/\Card{X},
\end{align*}
and for a function $f : M \To \Real$ that is finite on $X \setminus \{x\},$
the \noemph{corresponding punctured mean} is
\begin{align*}
\Integral_X^{[x]} f &:= \int_M f(y)\,d\NMeasure_X^{[x]}(y) 
= \frac{1}{\Card{X}} \mathop{\sum_{y \in X}}_{y \neq x} f(y).
\end{align*}
Note the division by $\Card{X}$ rather than $\Card{X}-1.$

The kernel $\Potential^{(s)}\big(\dist(x,y)\big) = \dist(x,y)^{-s}$ is called the \noemph{Riesz $s$-kernel.}
For a point $x \in M,$ define the function $\Potential_x^{(s)} : M \setminus \{x\} \To \Real$ as
\begin{align*}
\Potential_x^{(s)}(y) &:= \Potential^{(s)}\big(\dist(x,y)\big). 
\end{align*}
The mean Riesz $s$-potential at $x$ with respect to $M$ is then
\begin{align}
\label{Mean-potential-def}
\MeanPotential_M^{(s)}(x) &= \Integral_M \Potential_x^{(s)}, 
\end{align}
and the normalized energy of the Riesz $s$-potential on $M$ is
\begin{align*}
\Energy_M \Potential^{(s)} &= \Integral_M \MeanPotential_M^{(s)} 
= \int_M \int_M \dist(x,y)^{-s}\,d\NMeasure_M(y)\,d\NMeasure_M(x).
\end{align*}

For an $M$-code $X$, 
the \noemph{mean Riesz $s$-potential at $x$ with respect to $X$ but excluding $x$} is
\begin{align*}
\MeanPotential_X^{(s)}(x) &:= \Integral_X^{[x]} \Potential_x^{(s)}, 
\end{align*}
the normalized energy of the Riesz $s$-potential on $X$ is
\begin{align*}
\Energy_X \Potential^{(s)} 
&= 
\Integral_X \MeanPotential_X^{(s)} 
= 
\frac{1}{\Card{X}^2} \sum_{x \in X} \mathop{\sum_{y \in X}}_{y \neq x} \dist(x,y)^{-s},
\end{align*}
and the \noemph{mean on $X$ of the mean Riesz $s$-potential} is
\begin{align*}
\Integral_X \MeanPotential_M^{(s)} 
&= 
\frac{1}{\Card{X}} \sum_{x \in X} \int_M \dist(x,y)^{-s}\,d\NMeasure_M(y).
\end{align*}

The following bound is used in Lemma~\ref{Mean-potential-is-continuous-lemma} below
to prove the continuity of the mean Riesz $s$-potential.
\begin{lemma}\label{Small-ball-energy-estimate-lemma}
Let $M$ be a compact connected $d$-dimensional Riemannian manifold.
Then for the radius $R_{1}$ as per Lemma~\ref{Small-ball-estimate-lemma}, 
there is a constant $C_{3}$ such that for all $x \in M$ and $r \in (0,R_{1})$,
the normalized integral of the function $\Potential_x^{(s)}$ is bounded as
\begin{align}
\Integral_{\Ball(x,r)} \Potential_x^{(s)} &\leqslant C_{3} r^{d-s}.
\label{eq-small-ball-energy-estimate}
\end{align}
\end{lemma}
\begin{proof} ~

Fix $x \in M,$ and let $\Volume_M(r) := \NMeasure_M\big(\Ball(x,r)\big).$
Then for $r \in (0,R_{1}),$ the following equations and inequality hold, 
\begin{align*}
\Integral_{\Ball(x,r)} \Potential_x^{(s)}
&=
\int_{\Ball(x,r)} \dist(x,y)^{-s}\,d\NMeasure_M(y)
=
\int_0^r t^{-s}\,d\Volume_M(t)
\\
&= r^{-s} \Volume_M(r) + s \int_0^r t^{-s-1} \Volume_M(t)\,d t
\\
&\leqslant
C_{H} r^{d-s} + s \int_0^r C_{H} t^{d-s-1}\,dt
=
C_{H} \frac{d}{d-s}\ r^{d-s},
\end{align*}
where the inequality is a result of Lemma~\ref{Small-ball-estimate-lemma}.
Thus the estimate~\eqref{eq-small-ball-energy-estimate} is satisfied for
$C_{3} = C_{H}\ d/(d-s).$

\end{proof}

The proof of Theorem~\ref{manifold-discrepancy-separation-energy-theorem}
uses the continuity of the mean Riesz $s$-potential, as shown by the following lemma.
\begin{lemma}\label{Mean-potential-is-continuous-lemma}
Let $M$ be a compact connected $d$-dimensional Riemannian manifold.
Then for $s \in (0,d),$ the mean Riesz $s$-potential $\MeanPotential_M^{(s)}$ 
defined by~\eqref{Mean-potential-def} is continuous on $M.$
\end{lemma}
\begin{proof} ~

We show that the mean Riesz $s$-potential $\MeanPotential_M^{(s)}$ is continuous 
by using the method of proof of Kellogg~\cite[p. 150-151]{Kel29}.

Let $x \in M$ and recall that $\MeanPotential_M^{(s)}(x) = \Integral_M \Potential_x^{(s)}.$
Let $x'$ be another point of $M$ and consider the ball $\Ball'_r := \Ball(x',r),$
for some $r \in (0,R_{1}/3)$ where $R_{1}$ is a suitable radius 
as per Lemma~\ref{Small-ball-estimate-lemma}.
Consider $\MeanPotential_{\Ball'_r}^{(s)}(x) := \Integral_{\Ball'_r} \Potential_x^{(s)}.$
Since $\Potential_x^{(s)} > 0,$ it is always the case that $\MeanPotential_{\Ball'_r}^{(s)}(x) \geqslant 0.$
Either $\dist(x,x') \leqslant 2 r,$ in which case $x' \in \Ball(x,2 r)$ so that 
\begin{align*}
\Integral_{\Ball'_r} \Potential_x^{(s)}
&<
\Integral_{\Ball(x,3 r)} \Potential_x^{(s)}
\leqslant
3^{d-s}\ C_{3}\ r^{d-s}
\end{align*}
as per Lemma~\ref{Small-ball-energy-estimate-lemma},
or $\dist(x,x') > 2 r,$ so that
\begin{align*}
\Integral_{\Ball'_r} \Potential_x^{(s)}
&\leqslant
r^{-s}\ C_{H} r^d
=
C_{H} r^{d-s},
\end{align*}
as per Lemma~\ref{Small-ball-estimate-lemma}.
Therefore $\MeanPotential_{\Ball'_r}^{(s)} \To 0$ uniformly on $M$ as $r \To 0$.

So, given $\epsilon > 0$ we can take $r$ small enough that 
$\MeanPotential_{\Ball'_r}^{(s)}(x) < \epsilon/2$ for all $x \in M,$
and therefore  $\MeanPotential_{\Ball'_r}^{(s)}(x') < \epsilon/2,$ so
\begin{align*}
\Abs{\Integral_{\Ball'_r} \left(\Potential_x^{(s)} - \Potential_{x'}^{(s)} \right) } < \epsilon/2.
\end{align*}
With $\Ball'_r$ fixed, there is a distance $t>0$ such that when $\dist(x,x') \leqslant t,$ we have
\begin{align*}
\Abs{ \Potential_x^{(s)}(y) - \Potential_{x'}^{(s)}(y) }
&=
\Abs{ \dist(x,y)^{-s} - \dist(x',y)^{-s} }
\leqslant
\epsilon/2
\end{align*}
for all $y \in M \setminus \Ball'_r.$
In this case
\begin{align*}
\Abs{\Integral_{M \setminus \Ball'_r} \left(\Potential_x^{(s)} - \Potential_{x'}^{(s)} \right) } 
&\leqslant
\Integral_{M \setminus \Ball'_r} \Abs{\Potential_x^{(s)} - \Potential_{x'}^{(s)} } 
< 
\epsilon/2.
\end{align*}
Therefore $\Abs{\Integral_M \left(\Potential_x^{(s)} - \Potential_{x'}^{(s)} \right) } \leqslant \epsilon\ $ 
whenever $\dist(x,x') \leqslant t.$

\end{proof}

In fact, a stronger result holds, giving an estimate that is used in the proof of Theorem
\ref{manifold-discrepancy-separation-energy-theorem}.
\begin{lemma}\label{Mean-potential-is-Holder-continuous-lemma}
Let $M$ be a compact connected $d$-dimensional Riemannian manifold.
Then for $s \in (0,d),$ the mean Riesz $s$-potential $\MeanPotential_M^{(s)}$ 
defined by~\eqref{Mean-potential-def} is H\"older continuous on $M,$ with exponent $(d-s)/(d+1).$
Specifically, for $0 \leqslant t < \min \big( 1, (R_1/3)^{d+1} \big),$ the estimate
\begin{align}
\Abs{ \MeanPotential_M^{(s)}(x) - \MeanPotential_M^{(s)}(x') }
&=
\Order \left( t^{ (d-s)/(d+1) } \right)
\label{eq-Holder-estimate}
\end{align}
holds whenever $\dist( x, x' ) \leqslant t.$
\end{lemma}
\begin{proof} ~

We prove~\eqref{eq-Holder-estimate} by putting explicit estimates into the proof of 
Lemma~\ref{Mean-potential-is-continuous-lemma} above.
This proof therefore uses the notations and the definitions used there.

Firstly, the proof of Lemma~\ref{Mean-potential-is-continuous-lemma} establishes that 
for $r \in (0,R_{1}/3),$ where $R_{1}$ is a suitable radius 
as per Lemma~\ref{Small-ball-estimate-lemma},
\begin{align*}
\MeanPotential_{\Ball'_r}^{(s)}(x)
\leqslant
C_{4}\ r^{d-s},
\end{align*}
for all $x \in M,$ where $C_{4} := \max( C_H, 3^{d-s}\ C_{3} ).$
This yields the estimate
\begin{align}
\Integral_{\Ball'_r} \Abs{\Potential_x^{(s)} - \Potential_{x'}^{(s)} }
&= \Order(r^{d-s}).
\label{eq-inside-ball-estimate} 
\end{align}

Secondly, let $y \in M$ be such that $\dist(y,x') = r,$
with $0 < r < \min(1, R_{1}/3).$
If $\dist(x,x') = t < r,$ then by the triangle inequality, 
$\dist(x,y) \geqslant r-t,$ and so
\begin{align*}
\Abs{ \dist(x,y)^{-s} - \dist(x',y)^{-s} } &\leqslant (r-t)^{-s} - r^{-s}. 
\end{align*}
From the binomial expansion of $(r-t)^{-s}$ we have
\begin{align*}
(r-t)^{-s} - r^{-s} 
&= r^{-s} \big( (1 - t/r)^{-s} - 1 \big)
= r^{-s} \Order(t/r) = \Order(t r^{-s-1} ).
\end{align*}
We therefore have the estimate
\begin{align}
\Integral_{M \setminus \Ball'_r} \Abs{\Potential_x^{(s)} - \Potential_{x'}^{(s)} }
&= \Order(t r^{-s-1} ).
\label{eq-outside-ball-estimate} 
\end{align}

We can equate the orders of the estimates~\eqref{eq-inside-ball-estimate} and~\eqref{eq-outside-ball-estimate}
by setting $t := r^{d+1}.$
This yields the overall estimate
\begin{align*}
\Integral_{M} \Abs{\Potential_x^{(s)} - \Potential_{x'}^{(s)} }
&= \Order(r^{d-s} ) = \Order( t^{(d-s)/(d+1)} ),
\end{align*}
giving the result~\eqref{eq-Holder-estimate}.

\end{proof}
\begin{Remark}
The result~\eqref{eq-Holder-estimate} and is proof is split into two lemmas,
\ref{Mean-potential-is-continuous-lemma} and~\ref{Mean-potential-is-Holder-continuous-lemma},
to make the exposition easier to understand.
\end{Remark}

\section{Proof of Theorem~\ref{manifold-discrepancy-separation-energy-theorem}}\label{proof-section}

\begin{quiet-proof}

Fix the manifold $M$ and therefore fix $d.$
Fix $s \in (0,d),$ and drop all superscripts $(s)$ from the notation,
where this does not cause confusion.
Fix a sequence $\mathcal{X}$ having the required properties.
Fix $\ell,$ drop all subscripts $\ell,$
and examine the spherical code $X := \{x_1,\ldots,x_N\},$ so that $\Card{X}=N.$
The notation of the proof also uses the abbreviations $\MinDistBound := \MinDistBound(N),$
$\Disc := \Disc(N)$.

The first observation is that
\begin{align*}
\big(\Energy_X - \Energy_M\big)\,\Potential
&=
\Integral_X \MeanPotential_X - \Integral_M \MeanPotential_M
\\
&=
(\Integral_X \MeanPotential_X - \Integral_X \MeanPotential_M)
+
(\Integral_X \MeanPotential_M - \Integral_M \MeanPotential_M)
\\
&=
\Integral_X (\MeanPotential_X - \MeanPotential_M)
+
(\Integral_X - \Integral_M) \MeanPotential_M.
\end{align*}

The first part of the proof concentrates on the convergence to $0$ of the term
$\Integral_X (\MeanPotential_X - \MeanPotential_M).$
Since
\begin{align}
\Integral_X (\MeanPotential_X - \MeanPotential_M)
&=
\frac{1}{N} \sum_{x \in X} (\MeanPotential_X(x) - \MeanPotential_M(x))
\label{eq-potential-difference-sum} 
\end{align}
the proof proceeds by placing a uniform bound on the net mean potential 
$\MeanPotential_X(x) - \MeanPotential_M(x)$ at $x \in X$.
We express this net mean potential as a difference between 
Riemann-Stieltjes integrals, then integrate by parts.

Fix $x \in X.$ 
The volume of the ball $\Ball(x,r)$ with respect to the punctured
normalized counting measure $\NMeasure_X^{[x]}$ is
\begin{align*}
\Volume_X^{[x]} (r)
&:= 
\NMeasure_X^{[x]}\left(\Ball(x,r)\right)
=
\frac{\Card{\Ball(x,r) \cap X} - 1}{N}.
\end{align*}
Using $\Volume_M(r) := \NMeasure_M\big(\Ball(x,r)\big)$ to
denote the volume of $\Ball(x,r)$ with respect to the measure $\NMeasure_M,$
and integrating by parts, yields
\begin{align}
\MeanPotential_X(x) - \MeanPotential_M(x)
&=
\Integral_X^{[x]} \Potential_x - \Integral_M \Potential_x 
\notag
\\
&=
\int_M \Potential\left(\dist(x,y)\right)\, d\NMeasure_X^{[x]}(y)
-
\int_M \Potential\left(\dist(x,y)\right)\, d\NMeasure_M(y)
\notag
\\
&=
\int_0^{\infty} r^{-s}\, d\Volume_X^{[x]}(r)
-
\int_0^{\infty} r^{-s}\, d\Volume_M(r)
\notag
\\
&=
\int_0^{\infty} s r^{-s-1}\, \Volume_X^{[x]}(r)\, dr
-
\int_0^{\infty} s r^{-s-1}\, \Volume_M(r)\, dr
\notag
\\
&=
\int_0^{\infty} s r^{-s-1}\, \big(\Volume_X^{[x]}(r)-\Volume_M(r)\big)\, dr.
\label{eq-potential-difference-integral}
\end{align}

We now bound $\Abs{ \MeanPotential_X(x) - \MeanPotential_M(x) }$ by placing an upper bound on each of 
$\Volume_X^{[x]}(r)-\Volume_M(r)$ and its negative, $\Volume_M(r) - \Volume_X^{[x]}(r).$
%
%
%


Because the minimum distance between points of $X$ 
is bounded below by $\MinDistBound,$ each point of $X$ can be placed in a ball of radius $\MinDistBound/2,$
with no two balls overlapping.
Lemma~\ref{Ball-packing-argument-lemma} then implies that for $r < \diam(M),$
\begin{align*}
\Card{\Ball(x,r) \cap X} 
&\leqslant n_M(x,r+\MinDistBound/2,\MinDistBound/2) \leqslant C_{2}\,(r/\MinDistBound)^d,
\end{align*}
and so
\begin{align*}
\Volume_X^{[x]}(r)
&
\leqslant 
C_{2}\,\MinDistBound^{-d} N^{-1} r^d - N^{-1}.
\end{align*}

Since the normalized ball discrepancy $\Discrepancy(X)$ is bounded above by $\Disc,$
it is also true that for $0 < r < \diam(X),$
\begin{align*}
-\Disc 
&\leqslant 
\Volume_X^{[x]}(r) - \Volume_M(r) + N^{-1}
\leqslant
\Disc.
\end{align*}

Let $\RadiusH : = \gamma \Disc^{1/d}.$
Since $\Disc > N^{-1}$ as a result of \eqref{discrepancy-bound-def}, and since $X$ is well separated
with separation constant $\gamma$,
\begin{align*}
0 &< \MinDistBound < \RadiusH < \diam(M) ,
\end{align*}
for $N$ sufficiently large.
Since the minimum distance between points of $X$ 
is bounded below by $\MinDistBound,$
$\Volume_X^{[x]}(r) = 0$ when $r < \MinDistBound.$
Since $\NMeasure_M$ and $\NMeasure_X$ are probability measures on $M$,
$\Volume_M(r) = 1$ and $\Volume_X^{[x]}(r) = (N-1)/N$ when $r \geqslant \diam(M).$

The simple lower bound $\Volume_M(r) \geqslant 0$ for $0 < r \leqslant \RadiusH,$
and the bounds immediately above 
yield the following cases for the upper bound on $\Volume_X^{[x]}(r) - \Volume_M(r)$:
\begin{align*}
\Volume_X^{[x]}(r) - \Volume_M(r)
&\leqslant 
\begin{cases}
0, & r \in [0, \MinDistBound], 
\\
C_{2}\,\MinDistBound^{-d} N^{-1} r^d - N^{-1}, & r \in (\MinDistBound, \RadiusH),
\\
\Disc - N^{-1}, & r \in [\RadiusH,\diam(M)),
\\
-N^{-1}, & r \geqslant \diam(M).
\end{cases}
\end{align*}

Substitution back into~\eqref{eq-potential-difference-integral} results in the uniform upper bound
\begin{align*}
\MeanPotential_X(x) - \MeanPotential_M(x)
&=
\int_0^{\infty} s r^{-s-1}\, \big(\Volume_X^{[x]}(r)-\Volume_M(r)\big)\, dr
\\
&\leqslant
C_{2}\,\MinDistBound^{-d} N^{-1}\ s \int_{\MinDistBound}^{\RadiusH} r^{d-s-1}\, dr
\\
&\phantom{=}
+
\Disc \int_{\RadiusH}^{\diam(M)}\ s r^{-s-1}\, dr
\ - \ 
N^{-1} \int_{\MinDistBound}^{\infty} s r^{-s-1}\, dr
\\
&=
C_{2}\,\MinDistBound^{-d} N^{-1}\ \frac{s}{d-s} \left({\RadiusH}^{d-s} - {\MinDistBound}^{d-s}\right)
\\
&\phantom{=}
+
\Disc \left( {\RadiusH}^{-s} - {\diam(M)}^{-s} \right)
- N^{-1} {\MinDistBound}^{-s}.
\end{align*}

Noting that $\MinDistBound^d\,N = \gamma^d = \Theta(1),$
substituting in the value for $\RadiusH,$ 
and noting that $\Disc\,N > 1,$ 
results in the bound
\begin{align}
\MeanPotential_X(x) - \MeanPotential_M(x)
&\leqslant
\Order( {\RadiusH}^{d-s} )
+
\Order( {\Disc} {\RadiusH}^{-s} )
=
\Order( \Disc^{1-s/d} ).
\label{eq-upper-bound}
\end{align}


Arguments similar to those for the upper bound on 
$\Volume_X^{[x]}(r) - \Volume_M(r)$ result in the cases
\begin{align*}
\Volume_M(r) - \Volume_X^{[x]}(r)
&\leqslant
\begin{cases}
C_{H} r^d, & r \in [0, \RadiusH], 
\\
\Disc + N^{-1}, & r \in (\RadiusH,\diam(M)),
\\
N^{-1}, & r \geqslant \diam(M).
\end{cases}
\end{align*}

Substitution back into~\eqref{eq-potential-difference-integral} results in the uniform upper bound
\begin{align*}
\MeanPotential_M(x) - \MeanPotential_X(x) 
&=
\int_0^{\infty} s r^{-s-1}\, \big(\Volume_M(r)-\Volume_X^{[x]}(r)\big)\, dr
\\
&\leqslant
C_{H}\ s \int_0^{\RadiusH} r^{d-s-1}\, dr
+
\Disc \int_{\RadiusH}^{\diam(M)} s r^{-s-1}\, dr
+
N^{-1} \int_{\RadiusH}^{\infty} s r^{-s-1}\, dr
\\
&=
C_{H} \frac{s}{d-s} {\RadiusH}^{d-s}
+
\Disc \left( {\RadiusH}^{-s} - {\diam(M)}^{-s} \right)
+ N^{-1} {\RadiusH}^{-s}.
\end{align*}
Similarly to the argument for the upper bound on $\MeanPotential_X(x) - \MeanPotential_M(x)$, 
this gives the bound
\begin{align}
\MeanPotential_M(x) - \MeanPotential_X(x)
&\leqslant
\Order({\RadiusH}^{d-s}) + \Order(\Disc {\RadiusH}^{-s})   + \Order(N^{-1} {\RadiusH}^{-s})
=
\Order(\Disc^{1-s/d}).
\label{eq-lower-bound}
\end{align}

When the upper bounds~\eqref{eq-upper-bound}  
and~\eqref{eq-lower-bound} are combined, this results in the overall order estimate
\begin{align*}
\Abs{ \MeanPotential_X(x) - \MeanPotential_M(x) }
&=
\Order(\Disc^{1-s/d}).
\end{align*}
Therefore, recalling the sum~\eqref{eq-potential-difference-sum}, this shows that
\begin{align}
\Abs{ \Integral_X (\MeanPotential_X - \MeanPotential_M) }
&=
\Order(\Disc^{1-s/d}).
\label{eq-X-X-M-rate}
\end{align}

We now treat the convergence of the term $(\Integral_X - \Integral_M) \MeanPotential_M$ 
to $0$ as $N \to \infty.$

Since $\MeanPotential_M$ is continuous on $M$ 
as per Lemma~\ref{Mean-potential-is-continuous-lemma},
and since the sequence $\mathcal{X}$ is asymptotically equidistributed,
with each measure $\NMeasure_X$ being a probability measure on $M,$
by Bl\"umlinger's Theorem 2~\cite[p. 181]{Blu90},
the term $(\Integral_X - \Integral_M) \MeanPotential_M$ converges to $0$ as $N \to \infty.$
To obtain a rate of convergence for this term, we use Bl\"umlinger's Theorem 1~\cite[p. 180]{Blu90}
along with Lemma~\ref{Mean-potential-is-Holder-continuous-lemma}.

We adopt Bl\"umlinger's notation, and set
$f := \MeanPotential_M,$ $\Measure := \Measure_M,$ $\nu := \Measure(M) \NMeasure_X.$
We also adjust Bl\"umlinger's estimate to take into account that in our case
$\Measure(M)$ is not necessarily 1.
The estimate in this case is
\begin{align}
\Abs{ \nu( f ) - \Measure( f ) } &\leqslant T_1(r) + T_2(r) + T_3(r),
\label{eq-Blumlinger-Theorem-1-estimate}
\end{align}
where
\begin{align*}
T_1(r) &:= 
\norm{ f - f_r } \norm{ \nu },
\\
T_2(r) &:= 
\norm{ f } (\norm{ \nu } + \norm{ \Measure  }) 
\sup_{x \in M} \Abs{\frac{\Measure\big(\Ball(x,r)\big)}{\Volume_d(r)}-1},
\\
T_3(r) &:= 
\frac{\norm f }{ \Volume_d(r) } \int_M \Abs{ \nu\big( B(x,r) \big) - \Measure( B(x,r) \big) }\ d\, \Measure(x).
\end{align*}
The norm used here is $\norm{\cdot}_{\infty},$ the norm on $C(M).$
Therefore $\norm{ \nu } = \norm{ \Measure } = \Measure(M).$

We now estimate the order of each term with respect to $r$ and the discrepancy bound $\Disc.$
For $T_1(r)$ we find the extrema of $f$ and $f_r$ on $M.$
From Bl\"umlinger's definition of $f_r,$~\cite[p. 179]{Blu90}
we see that $f_r$ is the mean on $B(x,r)$ of $f$ with respect to $\Measure.$
It therefore holds that
\begin{align*}
\min_{x \in B(x,r)} f(x) &\leqslant f_r(x) \leqslant \max_{x \in B(x,r)} f(x).  
\end{align*}
Recalling that $f = \MeanPotential_M,$ and applying the estimate~\eqref{eq-Holder-estimate}
from Lemma~\ref{Mean-potential-is-Holder-continuous-lemma}, we obtain
\begin{align*}
\norm{ f - f_r } &= \Order( r^{(d-s)/(d+1)} ) 
\end{align*}
for $r$ sufficiently small.
Therefore
\begin{align}
T_1(r) &=  \Order( r^{(d-s)/(d+1)} ). 
\label{eq-T-1-estimate}
\end{align}

For $T_2(r),$ Bl\"umlinger's estimate~\eqref{eq-Blumlinger-estimate}
as per Lemma~\ref{Blumlinger-estimate-lemma} yields
$T_2(r) =  \Order( r^2 ).$

For $T_3(r),$ first note that
\begin{align*}
\frac{\norm{f}}{\Volume_d(r)} &= \Theta( r^{-d} ). 
\end{align*}
Since
$\Abs{ \nu\big( B(x,r) \big) - \Measure( B(x,r) \big) } \leqslant \Disc \Measure(M),$ 
this yields
\begin{align}
T_3(r) &= \Order( \Disc r^{-d} ). 
\label{eq-T-3-estimate}
\end{align}

To equate the orders of the estimates 
\eqref{eq-T-1-estimate} and~\eqref{eq-T-3-estimate} 
for $T_1(r)$ and $T_3(r),$
we set $r = \Disc^{(d+1)/(d^2 + 2 d - s)}.$
This results in an overall estimate of
\begin{align}
\Abs{ (\Integral_X - \Integral_M) \MeanPotential_M }
&=
\frac{1}{\Measure(M)} \Abs{ \nu( f ) - \Measure( f ) }
= \Order \big( \Disc^{(d-s)/(d^2 + 2 d - s)} \big).
\label{X-M-M-rate}
\end{align}

The estimates~\eqref{eq-X-X-M-rate} and~\eqref{X-M-M-rate} combine to
yield the estimate~\eqref{eq-manifold-discrepancy-separation-energy-estimate}.
\end{quiet-proof}

\section{Discussion}

Theorem~\ref{manifold-discrepancy-separation-energy-theorem} 
demonstrates the convergence of the normalized Riesz $s$-energy
of a well separated, equidistributed sequence of $M$-codes on 
a compact connected $d$-dim\-ensional Riemannian manifold $M$
to the energy given by the double integral of the normalized volume measure on $M,$
in the case where $0 < s < d.$
The estimated rate of convergence given by the theorem
is much slower than the corresponding rate of $\Disc^{1-s/d}$ on the sphere~\cite{Leo13}.

The proof of Theorem~\ref{manifold-discrepancy-separation-energy-theorem} relies on
the estimate~\eqref{eq-Blumlinger-Theorem-1-estimate} from
Bl\"umlinger's Theorem 1~\cite{Blu90}.
This, to some extent, resembles a Koksma-Hlawka-type inequality,
in that it contains three terms, each of which separate 
the dependence on the function and 
the dependence on the measure into different factors.
One key difference between the estimate~\eqref{eq-Blumlinger-Theorem-1-estimate} and
a Koksma-Hlawka-type inequality is that the term $T_3$ has $\Volume_d(r)$ in the denominator.
This makes it difficult to apply this estimate to the case of arbitrarily small positive $r.$

If the manifold $M$ actually had a Koksma-Hlawka-type inequality for the ball discrepancy $\Disc,$
with a function space $F_M$ containing the function $\MeanPotential_M,$
the estimate
\begin{align*}
\Abs{(\Integral_X - \Integral_M) \MeanPotential_M}
&\leqslant
\Disc\ V\big(\MeanPotential_M\big)
\end{align*}
would hold for some appropriate functional $V$ on the space $F_M.$
Unfortunately, not much is known about Koksma-Hlawka type inequalities for geodesic
balls on compact connected Riemannian manifolds, with the exception of
the sphere $\Sphere^d$~\cite[Section 3.2, p. 490]{BraD12}~\cite[Proposition 20]{BraSSW12}.

The papers by Brandolini et al.~\cite{BraCCGST10,BraCGT11} 
examine Koskma-Hlawka type inequalities on compact Riemannian manifolds.
The main results of those two papers concern discrepancies which are 
not in general the same as the geodesic ball discrepancy,
but they do suggest directions for further research.

Further research could address the following questions.
\begin{enumerate}
\item
For a compact connected Riemannian manifold $M$,
for what linear spaces $F_M$ does a Koksma-Hlawka type inequality 
\begin{align}
\Abs{(\Integral_X - \Integral_M) f}  
&\leqslant
\Discrepancy(X)\ V(f)
\label{eq-koksma-hlawka}
\end{align}
hold for all $f \in F_M,$
where the relevant discrepancy in the inequality is the geodesic ball discrepancy?
\item
What is the appropriate functional $V$ in~\eqref{eq-koksma-hlawka}?
Is $V$ a norm or a semi-norm on the function space $F_M$?
\item
For which compact connected Riemannian manifolds $M$ does the 
Koksma-Hlawka function space $F_M$ contain the mean potential function $\MeanPotential_M$?
\end{enumerate}

Finally, no mention has yet been made of constructions for, 
or even the existence of, well separated, admissible sequences on 
compact connected Riemannian manifolds.
The case of the unit sphere $\Sphere^d$ has been well studied~\cite{Leo13} and a 
number of constructions are known, including one that uses a partition of the sphere
into regions of equal volume and bounded diameter~\cite{Leo07Thesis}.

Damelin et al. have studied the discrepancy and energy of finite sets contained 
within measurable subsets of Hausdorff dimension $d$ 
embedded in a higher dimensional Euclidean space, 
where the energy and discrepancy are both
defined via an admissible kernel~\cite{DamHRZ10}.
One of their key results is to express the discrepancy of a finite set with respect to an 
equilibrium measure as the square root of the difference between 
the energy of the finite set and the energy of the equilibrium measure 
\cite[Corollary 10]{DamHRZ10}.
They have also studied the special case where both the measurable subset and
the kernel are invariant under the action of a group~\cite[Section 4.3]{DamHRZ10}.
This case includes compact homogeneous manifolds~\cite{DamLRS09}.

The methods of Damelin et al. might be used to prove the equidistribution of 
a sequence of $M$-codes $\mathcal{X^{\ast}},$ 
where each code $X_{\ell}^{\ast}$ has the minimum Riesz $s$-energy of 
all codes of cardinality $\Card{X_{\ell}^{\ast}}.$
Much care must be taken: although their definition of an admissible
kernel includes the Riesz $s$-kernels as defined in this paper~\cite[Section 2.1]{DamHRZ10}, 
their definitions and results are framed in terms of sets embedded in Euclidean space,
their definition of discrepancy is given in terms of a norm depending on the kernel~\cite[(8)]{DamHRZ10},
the measure used in their Corollary 10 is the equilibrium measure, not the uniform measure,
and their definition of energy includes the diagonal terms excluded in this paper,
so that the energy of the Riesz $s$-kernel on a finite set is infinite~\cite[(5) and Section 3]{DamHRZ10}.

Brandolini et al.~\cite[p. 2]{BraCCGST10} give an example 
where the existence of a partition of the manifold $M$ into $N$ regions,
each with volume $N^{-1}$ and diameter at most $c N^{-1/d},$ yields an $M$-code
$X$ obtained by selecting one point from each region, and this gives a bound
on the quadrature error of the code $X$ with respect to bounded functions on the manifold $M$.
Such a partition might be constructed by adapting the modified Feige-Schechtman 
partition algorithm for the unit sphere~\cite{FeiS02} 
\cite[3.11.4, pp. 145-148]{Leo07Thesis}.  
Care must be taken to adapt the algorithm, in particular to choose an appropriate
radius for the initial saturated packing of the manifold $M$ by balls of a fixed radius.
Also, it would need to be proven that the adapted algorithm works for all 
compact connected Riemannian manifolds and all cardinalities $N$.

A recent paper by Ortega-Cerd\`a and Pridhnani~\cite{OrtP12} treats
sequences of Fekete point sets on some types of smooth compact connected
Riemannian manifolds $M,$ showing that such sequences are 
uniformly separated~\cite[Theorem 9]{OrtP12}  and 
asymptotically equidistributed~\cite[Theorem 11]{OrtP12}.
Uniform separation~\cite[p. 2106]{OrtP12} is defined in terms of the orthonormal basis for $L^2(M)$ 
consisting of eigenfunctions of the Laplacian operator on $M$, as opposed to the
purely geometric concept of well-separation used in this paper.
Fekete point sets are defined by maximizing the determinant of a Vandermonde matrix defined by
the values of each of the basis eigenfunctions at each of the points~\cite[Definition 8]{OrtP12}.
These point sets are therefore determined numerically by using optimization methods rather than by
construction.

Clearly, further research is needed to address the 
construction on compact connected Riemannian manifolds of sequences of 
point sets that are both equidistributed and well-separated.

\subsection*{Acknowledgements}
The current work was conducted at ANU and preliminary findings were
presented at the 3rd DWCAA at Canazei in 2012.
The support of the Australian Research Council 
under its Centre of Excellence program is gratefully acknowledged.
Thanks to Ed Saff, who posed the original problem on the sphere.
Thanks to Stefano De Marchi, Alvise Sommariva and Marco Vianello 
for hosting the author on a visit to the University of Padova in 2012.
Thanks to Leonardo Colzani and Giacomo Gigante 
for an updated version of the joint paper~\cite{BraCCGST10}.
Thanks also to Lashi Bandara, Johann Brauchart, Julie Clutterbuck, Thierry Coulhon, 
Mathew Langford, Ed Saff and David Shellard for valuable discussions.
One of my anonymous reviewers suggested improvements and simplifications of the arguments
given in an earlier draft of this paper, and encouraged me to
``look for a more quantitative result.''
This prompted me to re-examine Bl\"umlinger's paper~\cite{Blu90}
and take advantage of the estimates given there in Theorem 1,
resulting in the current version of Theorem \ref{manifold-discrepancy-separation-energy-theorem}
and its proof.
%
\noindent

%
\end{document}